\documentclass[letterpaper, 10 pt, conference]{ieeeconf}
\IEEEoverridecommandlockouts

\usepackage[pdftex]{graphicx}
\usepackage{algorithmic}
\usepackage{array}
\usepackage{color}
\definecolor{mygreen}{RGB}{10, 128, 128}
\definecolor{mybro}{RGB}{178, 34, 34}
\definecolor{myred}{RGB}{150, 24, 24}

\usepackage{cite}
\usepackage{amsmath}
\usepackage{amssymb}
\usepackage{bbm}
\usepackage[english]{babel}
\usepackage{dsfont}
\usepackage{makeidx}
\usepackage{bbold}
\usepackage{latexsym,remreset}
\usepackage{multirow}
\usepackage{amsfonts}
\usepackage{pgfplots}
\usepackage{mathtools}
\usepackage{textcomp}
\usepackage{scrextend}
\usepackage{commath}
\usepackage{graphicx}
\usepackage{url}
\usepackage{enumerate}
\usepackage{caption}
\usepackage{subcaption}

\newtheorem{theorem}{Theorem}

\newtheorem{definition}{Definition}
\newtheorem{lemma}{Lemma}
\newtheorem{proposition}{Proposition}

\newtheorem{remark}{Remark}

\makeatletter
\let\NAT@parse\undefined
\makeatother

\usepackage[final]{hyperref}
\hypersetup{
	colorlinks   = true, %Colours links instead of ugly boxes
	urlcolor     = myred, %Colour for external hyperlinks
	linkcolor    = mygreen, %Colour of internal links
	citecolor   = mybro %Colour of citations
}

\title{\LARGE \bf A Fast Certificate for Power System Small-Signal Stability}
\author{Amin~Gholami and Xu~Andy~Sun % <-this % stops a space
\thanks{The authors are with the H. Milton Stewart School of Industrial and Systems Engineering, Georgia Institute of Technology, Atlanta, GA 30332 USA (e-mail: {\tt\small a.gholami@gatech.edu}; {\tt\small andy.sun@isye.gatech.edu}).}% <-this % stops a space
}

\begin{document}	
\maketitle

\thispagestyle{plain}
\pagestyle{plain}

%\thispagestyle{empty}
%\pagestyle{empty}

\begin{abstract}
Swing equations are an integral part of a large class of power system dynamical models used in rotor angle stability assessment. Despite intensive studies, some fundamental properties of lossy swing equations are still not fully understood. In this paper, we develop a sufficient condition for certifying the stability of equilibrium points (EPs) of these equations, and illustrate the effects of damping, inertia, and network topology on the stability properties of such EPs. The proposed certificate is suitable for real-time monitoring and fast stability assessment, as it is purely algebraic and can be evaluated in a parallel manner. Moreover, we provide a novel approach to quantitatively measure the degree of stability in power grids using the proposed certificate. Extensive computational experiments are conducted, demonstrating the practicality and effectiveness of the proposal.
%
%We also derive the conditions under which the hyperbolicity of EPs breaks and different types of bifurcations occur. %Specifically, we provide sufficient and necessary conditions for Hopf bifurcation, and further characterize the bifurcated periodic orbit through center manifold analysis.
%Throughout the paper, we discuss both lossless and lossy (i.e., with transfer conductance) models of classical swing equations. 
%Our findings provide new insights into the dynamic behavior and oscillatory solutions of power systems. %We demonstrate different aspects of our findings using numerical simulations.
\end{abstract}

%\begin{IEEEkeywords}
%Power system stability, swing equation, small-signal stability.
%\end{IEEEkeywords}

% Use this to place sponsorships
%\thanksto{Applicable sponsors, if any, should be placed using the \emph{thanksto} command.}

\section{Introduction}
%\subsection{Motivation}
Power system stability has been an important topic in power engineering for many years. There has been continuing advancement in the understanding of the stability issues of the system. In the recent decade, the proliferation of renewable energy resources has added new dimensions to the problem. The uncertainty and volatility of these resources have brought about significant stochastic transitions from one operating point to another \cite{2010-Conejo-decision-making}, thereby making the system more prone to instability.
% With the proliferation of renewable energy resources, future power grids will be more complex and prone to instability. The uncertainty and volatility pertaining to these resources will bring about radical and stochastic transitions from one operating point to another.
%changes in the operating point of the system.
%Maintaining the stability of the system will be, therefore, one of the main challenges of future grids.
%
%

Owing to the complexity and high dimensionality of power systems, several CIGRE and IEEE Task Forces have classified power system stability into appropriate categories with the aim of facilitating the assessment of the problem  \cite{2004-stability-classification}. 
%The classification is commonly based on the size of the disturbance, the duration of the phenomenon, the components which influence the phenomenon, and the main system variable in which instability can be observed \cite{2004-stability-classification}. 
In each category, a set of simplifying assumptions are made and an appropriate system model with a reasonable level of details is adopted. One of the most fundamental models used in several categories of stability (especially rotor angle stability) is the swing equation model. This model describes the nonlinear relation between the power output and voltage angles of synchronous generators and can be used to analyze the short term dynamical behaviour of the system. 
%Despite its simple-looking form, which will be given shortly, the dynamics of swing equations is extraordinarily rich and its application is ubiquitous in power system dynamical models and analysis. %(usually $3$ to $5$ seconds following the disturbance).

The application of swing equations is not restricted to the characterization of interconnected synchronous machines. They can also be used to model the behavior of inverter-based resources, which can be controlled to emulate the behavior of synchronous machines \cite{2016-Xie-multi-microgrid}.
% 2- Its application goes beyond power system stability analysis, and it appears in other deciplines such as physics Kuromoto
% 3- EVEN DERs
Despite such a wide range of applications, 
%Although swing equations have been extensively used in various contexts for a long time,
some basic questions on the equilibrium points (EPs) of swing equations are not fully understood. In particular,
\begin{enumerate}[(i)]
    \item Under what conditions an EP of swing equations with nontrivial transfer conductance is asymptotically stable? \label{Q1}
    \item  What is the relation between the network structure of a power system and the stability of the EPs of swing equations? \label{Q2}
\end{enumerate}
Such challenging questions have perplexed many researchers over the years, and some parts of the puzzle have been solved. For instance, the EPs of swing equations with zero transfer conductance (the so-called \emph{ lossless} model) have been studied in the 1980s (see e.g. Chiang \textit{et al.} \cite{1988-Chiang-stability-regions} and Zaborszky \textit{et al.} \cite{ 1988-Zaborszky-phase-portrait}.) They assume that there is a unique stable EP and a finite number of unstable EPs in any $2\pi$ interval of generator angle coordinate. It is shown that the stability boundary of a stable EP consists of the stable manifolds of all the EPs (and/or closed orbits) on the stability boundary. Moreover, various methods in the broad category of the so-called direct methods have been developed to estimate the region of attraction of EPs \cite{2011-Chiang-book-direct-methods,1985-Varaiya-direct-methods}. These methods not only avoid expensive time-domain integration of swing equations, but also provide a quantitative measure of the degree of stability. Unfortunately, the existing methods are mostly limited to lossless systems and require a significant computational effort.
 More recently, the authors in \cite{2016-Turitsyn-Lyapunov-Functions} have alleviated some of these drawbacks.
%by showing that there exists a convex set of Lyapunov functions certifying the transient stability of a given power system. They propose an algorithm to find the best suited Lyapunov function in the family to specific
%contingency situations.

The characteristics of swing equations with nontrivial transfer conductance (the so-called \emph{lossy} model) are more challenging to analyze. This is partly due to the fact that there is no global energy function for such systems \cite{1989-Chiang-energy-functions-lossy}, and therefore, some main approaches (e.g., the energy function method) to investigate these equations cannot be directly applied. Nonetheless, several approaches are devised over the years. For instance, reference \cite{1979-Athay-practical-method} computes numerical energy functions to deal with the effects of transfer conductances on the system behavior.
% More recently, in \cite{2005-Silva-Smooth-perturbation-LossyNetworks}, the authors wrote a follow-up paper on \cite{1979-athay-practicalMethod-lossyNetworks} and showed that this approximated energy-like function is neither a Lyapunov function in the usual sense, nor an
% extended Lyapunov function, when the transfer conductances are taken into account. In spite of that, a function attending the requirements of the extension of the \textit{Invariance Principle} (the LaSalle's invariance principle in our previous discussions), that is, an extended Lyapunov function, can be obtained by smooth perturbations on that energy-like function. This perturbed function can be used to estimate the attraction area without approximations or conjectures. Indeed, the difference between the proposed extended
% Lyapunov function and the approximated energy-like function has the order of a smooth perturbation.
In \cite{2005-ortega-nontrivial-transfer-conductances2}, the authors extend the lossy swing equation model by considering the dynamics of the excitation system, and ensure the asymptotic stability of the operating points by designing a nonlinear feedback control for the generator excitation field. In \cite{1980-Skar-stability-thesis}, the local stability of swing equations with nontrivial transfer conductance is examined by linearization and conditions for stability of EPs are established. It is found that undamped swing equations can be stable only under very special circumstances.  
Another set of literature that address similar questions are the recent studies of the synchronization of Kuramoto oscillators that
are applicable to the stability analysis of lossy swing equations with strongly overdamped generators \cite{2012-Dorfler-Kuramoto}. Furthermore, exploring question (\ref{Q2}), the recent work \cite{2019-Milano-Topology-Impact}
statistically studies the impact of topology of the network on transient stability.
%
% In 1981, Bergen and Hill \cite{1981-Bergen-Hill-Structure-Preserving} suggested a way to circumvent the introduction of transfer conductance in the swing equations. The key idea of their model is an assumption of frequency-dependent load power, rather than the usual impedance loads which are subsequently absorbed into a reduced network. Similarly, in \cite{1985-Tsolas-Structure-Preserving-Energy}, loads are modeled as PQ buses in a network preserving model. Then a novel energy function is proposed which includes additional terms corresponding to the energy stored in the loads and field winding. A characterization of the stability region is derived based on this energy function.
%
%In this paper, we aim to address questions (\ref{Q1}) and (\ref{Q2}), and provide a rigorous analysis of the stability and hyperbolicity of EPs for both lossy and lossless swing equation models. There are two main contributions in the present paper. 

In this paper, we aim to address questions (\ref{Q1}) and (\ref{Q2}), and provide a rigorous analysis of the stability of EPs in lossy swing equation models. There are two main contributions in the present paper. 
\begin{itemize}
    % \item We analytically investigate the linearized swing equations without choosing any reference machine. We believe such a reference-free model is more close to reality and does not require the existence of an infinite-bus, which is of paramount importance in the modern smart grids with distributed resources (see \cite{1999-Alberto-role-of-reference} on the role of reference machine in stability assessment).
    \item We characterize the relationship between the Jacobian of swing equations and the underlying graph of power grids. Specifically, we associate a weighted graph with the swing equation model and then mathematically describe the relationship between the spectrum of the graph Laplacian and the spectrum of the swing equation Jacobian.
    
    \item
    %We develop a set of sufficient conditions under which the swing equation EPs are stable. 
    %For lossless systems, the stability of EPs has been analyzed in the literature using the Lyapunov theorem (see e.g. \cite{1988-Zaborszky-phase-portrait}). Here, we provide a novel graph-theoretic proof which produces new insights for the dynamics of swing equations. 
     We develop a sufficient condition under which the EPs of lossy swing equations are stable.
    %a set of sufficient conditions under which the swing equation EPs are stable.
    In addition to providing new insights into the theory of stability, the derived conditions are easy to check, use only local information, and are suitable for real-time monitoring and fast stability assessment.
    The proposed stability certificate can be interpreted as enforcing an upper bound on the matrix norm of the Laplacian of the underlying graph of the system. We show that the aforementioned upper bound is proportional to the square of damping and inverse of inertia at each node of the power grid. These results provide new insights into the way the damping and inertia at each node of the system would affect the stability of EPs. We also illustrate how the proposed condition provides a quantitative measure of the degree of stability in power systems.
\end{itemize}

The rest of our paper is organized as follows. Section \ref{sec: Background} provides a brief background on dynamical systems and swing equations. In Section \ref{Sec: Linearization and Spectrum of Jacobian}, the swing equation model is linearized and the linkage between the Jacobian of swing equations and the underlying graph of the power grid is established. Section \ref{Sec: Stability and Hyperbolicity of the Equilibrium Points} is devoted to the main results on the stability of the swing equation EPs. Section \ref{Sec: Computational Experiments} further illustrates the developed analytical results through numerical examples, and finally, the paper concludes with Section \ref{Sec: Conclusions}.
%---------------------------------------------------------------------
\section{Background} \label{sec: Background}
\subsection{Notations}
We use $\mathbb{C_{-}}$ to denote the set of complex numbers with negative real part, and $\mathbb{C_{0}}$ to denote the set of complex numbers with zero real part. $j=\sqrt{-1}$ is the imaginary unit, which should not be confused with the subscript $j$ that is used as an index. The spectrum of a matrix $A\in\mathbb{R}^{n\times n}$ is denoted by $\sigma(A)$.
%--------------------------------------
\subsection{Autonomous Ordinary Differential Equations}
Suppose $f:\mathbb{R}^n\to \mathbb{R}^n$ is a smooth vector field, where the term smooth here means continuously differentiable. An autonomous ordinary differential equation (ODE) is an equation of the form
\begin{align} \label{eq: nonlinear ode}
    \dot{x}=f(x),
\end{align}
where the dot denotes differentiation with respect to the independent variable $t$ (here a measure of time), and the dependent variable $x$ is a vector of state variables. If $f(x_0)=0$ for some $x_0\in\mathbb{R}^n$, then $x_0$ is called an EP. 
%The smoothness of the vector field $f$ is a sufficient condition for existence and uniqueness of solution.
%
Let us define the function $\phi:\mathbb{R} \times \mathbb{R}^n \to \mathbb{R}^n$ as follows: For any $x\in \mathbb{R}^n$, let $t\mapsto\phi(t,x)$ be the solution of the ODE \eqref{eq: nonlinear ode}, that is, $\frac{d\phi}{dt}(t,x)=f(\phi(t,x)), \forall t \in \mathbb{R}$. Moreover, $\phi(0,x)=x$.  
    Now, an EP $x_0$ of the ODE \eqref{eq: nonlinear ode} is 
    \begin{itemize}
        \item stable (in the sense of Lyapunov) if for each $\epsilon>0$, there exits a number $\xi >0$ such that $||\phi(t,x)-x_0||<\epsilon, \forall t\ge 0$ whenever $||x-x_0||<\xi$;
        \item unstable if it not stable;
        \item asymptotically stable if it is stable and $\xi$ can be chosen such that $\lim_{t\to\infty} ||\phi(t,x)-x_0||=0$ whenever $||x-x_0||<\xi$.
    \end{itemize}
Recall that if $x_0$ is an EP for the ODE \eqref{eq: nonlinear ode} and if all eigenvalues of the linear transformation $\nabla f(x_0)$ have negative real parts, then $x_0$ is asymptotically stable.    
\subsection{Multi-Machine Swing Equations} \label{subsec: Multi-Machine Swing Equations}
Consider a power system with the set of generators $\mathcal{N}=\{1,\cdots,n\}, n\in\mathbb{N}$. Based on the classical small-signal stability assumptions \cite{2018-Sauer-Dynamics}, the mathematical model for a power system is described by the following system of nonlinear autonomous ODEs, aka swing equations:  
\begin{subequations} \label{eq: swing equations}
	\begin{align}
	& \dot{\delta}_i(t) = \omega_i(t), && \forall i \in \mathcal{N},  \label{eq: swing equations a}\\
	&  \frac{M_i}{\omega_s} \dot{\omega}_i(t)+ \frac{D_i}{\omega_s} \omega_i(t) = P_{m_i} - P_{e_i}(\delta(t)), && \forall i \in \mathcal{N}, \label{eq: swing equations b}
	\end{align}
\end{subequations}
where for each generator $i\in\mathcal{N}$, $P_{m_i}$ and $P_{e_i}$ are respectively the mechanical and electrical power in per unit, $M_i$ is the inertia constant in seconds, $D_i$ is the unitless damping coefficient, $\omega_{s}$ is the synchronous angular velocity in electrical radians per seconds, $t$ is the time in seconds, $\delta_i(t)$ is the rotor electrical angle in radians, and finally $\omega_i(t)$ is the deviation of the rotor angular velocity from the synchronous velocity in electrical radians per seconds. Henceforth we do not explicitly write the dependence of the state variables $\delta$ and $\omega$ on time $t$. 
The electrical power $P_{e_i}$ in \eqref{eq: swing equations b} is given by: 
\begin{align} \label{eq: flow function}
	P_{e_i}(\delta) & = \sum \limits_{j = 1}^n { V_i  V_j Y_{ij} \cos \left( \theta _{ij} - \delta _i + \delta _j \right)},
% 	& = \sum_{j}^n E_i E_j \left( B_{ij} \sin ( \delta_i - \delta_j ) + G_{ij} \cos (\delta_i - \delta_j) \right) && \forall i \in \{1,...,n\}, 
\end{align}
where $V_i$ is the terminal voltage magnitude of generator $i$, and $Y_{ij}\measuredangle \theta_{ij}$ is the $(i,j)$ entry of the reduced admittance matrix.
% where $Y=Y\in\mathbb{C}^{n\times n}$ is the reduced admittance matrix
%  and $Y_{ij} = | Y_{ij} | \measuredangle \theta_{ij}$, $ E_i = \left| E_i \right| \measuredangle \delta_i$. 
% Moreover, $B_{ij}:=| Y_{ij} | \sin(\theta_{ij})$ and $G_{ij}:=| Y_{ij} | \cos(\theta_{ij})$ are susceptances and conductances obtained from the reduced admittance matrix. 
\begin{definition}[flow function] \label{def: flow function}
The smooth function $P_e:\mathbb{R}^n \to \mathbb{R}^n$ given by $\delta \mapsto P_e(\delta)$ in \eqref{eq: flow function} is called the flow function.
\end{definition}
Since the flow function is smooth, there exists a unique solution to the swing equation \eqref{eq: swing equations}. 
The flow function is invariant to the translation of $\delta \mapsto \delta + \alpha \mathbf{1}$, where $\alpha \in \mathbb{R}$ and $\mathbf{1}\in\mathbb{R}^n$ is the vector of all ones, i.e., $P_e(\delta + \alpha \mathbf{1})=P_e(\delta)$. A common way to deal with this situation is to define a reference bus and refer all other bus angles to it. This is equivalent to projecting the original state space onto a lower dimensional space. 
%However, this projection usually requires the existence of an infinite bus (e.g., a large power plant) \cite{2018-Sauer-Dynamics}. In this paper, we do not choose any reference bus and we do not require the existence of an infinite bus. 
% In the following subsections, we will discuss the two common methods of choosing the reference bus and the approximation behind each of these methods. Our aim in this manuscript, however, is to not follow this convention, and rather study the system in the original state space. This preference is because of the symmetrical properties of the original state space $\mathcal{S}$ that will be lost after the projection. 
\section{Linearization and Spectrum of Jacobian} \label{Sec: Linearization and Spectrum of Jacobian}
\subsection{Linearization}
%Let us take the state variable vector $[\delta,\omega]\in\mathbb{R}^{2n}$ into account and note that the first step in studying the stability of swing equations is to analyze 
The Jacobian of the vector field in \eqref{eq: swing equations} is given by
\begin{align}\label{eq: J}
J := \begin{bmatrix}
0 & I \\
-M^{-1} L  & - M^{-1}D \\
\end{bmatrix} \in\mathbb{R}^{2n\times 2n},
\end{align}
where $I\in\mathbb{R}^{n\times n}$ is the identity matrix, $M= \frac{1}{\omega_s}  \mathbf{diag}(M_1,\cdots,M_n)$, and $D=\frac{1}{\omega_s}\mathbf{diag}(D_1,\cdots,D_n)$. Moreover, $L\in\mathbb{R}^{n\times n}$ is the Jacobian of the flow function with the diagonal entries:
\begin{align*} 
\frac {\partial P_{e_i}} {\partial \delta_i} = \sum \limits_{j \ne i} { V_i V_j Y_{ij} \sin \left( {\theta _{ij} - {\delta _i} + {\delta _j}} \right) }, \forall i \in \mathcal{N},
% & = \sum \limits_{j \ne i} { V_i V_j Y_{ij} \left(\sin(\theta _{ij})\cos(\delta _i-\delta _j) - \cos(\theta_{ij})\sin(\delta_i - \delta_j) \right) }   \\
% & = \sum \limits_{j \ne i}  V_i V_j B_{ij} \cos(\delta _i-\delta _j) -  \sum \limits_{j \ne i}  V_i V_j G_{ij} \sin(\delta_i - \delta_j)                \\
% &= [L^0]_{ii} -  \sum \limits_{j \ne i}  V_i V_j G_{ij} \sin(\delta_i - \delta_j) =[L^0]+L_{ii},
\end{align*}
and off-diagonal entries 
\begin{align*}
\frac{\partial P_{e_i} } {\partial \delta _j} =  - {V_i} {V_j}{Y_{ij}}\sin \left( {{\theta _{ij}} - {\delta _i} + {\delta _j}} \right),\forall i,j \in \mathcal{N}, j \neq i.
\end{align*}
% \begin{align} \label{eq: J1 non-diagonals}
%  \\
% & = - V_i V_j Y_{ij} \left(\sin(\theta _{ij}) \cos(\delta _i-\delta _j) - \cos(\theta_{ij})\sin(\delta_i - \delta_j) \right) \\
% & =   - V_i V_j B_{ij} \cos(\delta _i-\delta _j) +  V_i V_j G_{ij}  \sin(\delta_i-\delta_j)      \\
% & =   [L^0]_{ij} +  V_i V_j G_{ij}  \sin(\delta_i-\delta_j) = [L^0]_{ij} +  K_{ij}. 
% \end{align}
%The entries of $L$ are: (define $G_{ij}:=Y_{ij}\cos (\theta_{ij})$ and $B_{ij}:=Y_{ij}\sin (\theta_{ij})$)
% ---- trimmed
% The matrix $L$ plays a prominent role in the spectrum of the Jacobian matrix $J$. 
In the following subsection, we study the role of matrix $L$ in the spectrum of the Jacobian matrix $J$ .
%-----------------------------------------------------------------
\subsection{Spectral Relationship Between Matrices $J$ and $L$}
We establish the spectral relationship between $J$ and $L$ via a singularity constraint. Let us first define the concept of a quadratic matrix pencil \cite{2001-Tisseur-Pencil}. Consider $n\times n$ real matrices $Q_0,Q_1,$ and $Q_2$. A quadratic matrix pencil is a matrix-valued function $P:\mathbb{C}\to\mathbb{R}^{n \times n}$ given by $\lambda \mapsto P(\lambda)$ such that  $P(\lambda) = \lambda^2Q_2 + \lambda Q_1 + Q_0$.
\begin{lemma} \label{lemma: relation between ev J and ev J11}
	$\lambda$ is an eigenvalue of $J$ if and only if the quadratic matrix pencil $P(\lambda):= \lambda^2 M + \lambda D + L$ is singular.
\end{lemma}
\begin{proof}
	Let $\lambda$ be an eigenvalue of $J$ and $[v_1 , v_2]$ be a corresponding eigenvector. Then	
	\begin{align} \label{eq: J cha eq}
	\begin{bmatrix}
	0 & I \\
	-M^{-1} L     &     - M^{-1}D \\
	\end{bmatrix}   \begin{bmatrix} v_1 \\v_2  \end{bmatrix}  = \lambda   \begin{bmatrix} v_1 \\v_2  \end{bmatrix}.
	\end{align}
	Thus, $ v_2 = \lambda v_1$ and $ M^{-1} L + \lambda ( M^{-1} D  + \lambda I  ) )   v_1 = 0$, which implies
    \begin{align} 
% 	& \left(     M^{-1} L + \lambda ( M^{-1} D  + \lambda I  )       \right)   v_1 = 0 \\
% 	%& \Rightarrow \;  M^{-1}\left( L + \lambda D +  \lambda^2 M \right) v_1 = 0 \\
% 	& \Rightarrow \; 
\left( L + \lambda D +  \lambda^2 M \right) v_1 = 0. \label{eq: quadratic matrix pencil}
	\end{align}  
	Since the eigenvector $v$ is nonzero, we have $v_1 \not = 0$ (otherwise $v_2 = \lambda \times 0 = 0 \implies v = 0 $). Equation \eqref{eq: quadratic matrix pencil} implies that the matrix pencil $P(\lambda)= \lambda^2 M + \lambda D + L$ is singular.
	Conversely, suppose there exists $\lambda \in \mathbb{C}$ such that  $P(\lambda)= \lambda^2 M + \lambda D + L$ is singular. Choose a nonzero $v_1 \in   \mathbf{ker}(P(\lambda))$ and let $ v_2 := \lambda v_1$. 
	Accordingly, the characteristic equation \eqref{eq: J cha eq} holds, and consequently, $\lambda$ is an eigenvalue of $J$.
\end{proof}
% \begin{lemma}  \label{lemma: matrix form of pencil sinularity}
% If $\lambda = \alpha + j \beta$, then $P(\lambda)$ is singular if and only if 
% \begin{align*}
% \mathcal{M}(\alpha,\beta):= \begin{bmatrix}
% 2(\alpha^2-\beta^2) H + \alpha D + L  & -\beta(4\alpha H + D)\\ \beta(4\alpha H + D) & 2(\alpha^2-\beta^2) H + \alpha D + L  
% \end{bmatrix}
% \end{align*}
% is singular.
% \end{lemma}

% \begin{proof}
% $P(\lambda) = 2(\alpha^2-\beta^2+2j\alpha\beta)H+(\alpha + j\beta) D + L = 2(\alpha^2-\beta^2) H + \alpha D + L + j\beta(4\alpha H + D)$. 
% We have
% %if $|\alpha|>|\beta|$ and $L$, then the real part of $P(\lambda)$ is positive definite. 
% \begin{align*}
% (A+jB)(u+jv) = (Au - Bv) + j(Av + Bu) = 0 \quad \Leftrightarrow \quad \begin{bmatrix}
% A & -B \\ B & A 
% \end{bmatrix} \begin{bmatrix}
% u \\ v
% \end{bmatrix} = 0.
% \end{align*}
% \end{proof}
Lemma \ref{lemma: relation between ev J and ev J11} illustrates the role of matrix $L$ in the stability of the EPs. {We note that \cite[Proposition 5.14]{2018-Dorfler-algebraic-graph-theory} presents similar results under symmetry assumptions, while  matrices $L$, $M$, and $D$ here in Lemma \ref{lemma: relation between ev J and ev J11} are not necessarily symmetric.
}
Next, we look more closely at the spectrum of $L$.

\subsection{Graph Induced by $L$ and Its Spectral Properties} \label{subsec: Spectral Properties of the Jacobian Matrix L}
The Jacobian $L$ of the flow function encodes the graph structure of the power network. To see this, we can define a weighted directed graph $\mathcal{G}=(\mathcal{N},\mathcal{A},\mathcal{W})$ where each node $i\in\mathcal{N}$ corresponds to a generator and each directed arc $(i,j)\in\mathcal{A}$ corresponds to the entry $(i,j), i\ne j$ of the Jacobian matrix $L$. We further define a weight for each arc $(i,j)\in\mathcal{A}$:
\begin{align} \label{eq: weights of digraph lossy}
w_{ij} = {V_i} {V_j}{Y_{ij}}\sin \left( \varphi_{ij} \right), \quad\forall (i,j)\in\mathcal{A},
\end{align}
where $\varphi_{ij} := {{\theta _{ij}} - {\delta _i} + {\delta _j}}$. With the above definitions, we can see that the Jacobian matrix $L$ of the flow function in \eqref{eq: J} is indeed the Laplacian of the directed graph $\mathcal{G}$ defined as $L = D^+(\mathcal{G}) - A(\mathcal{G})$,
% \begin{align} \label{eq: lablacian of lossless}
% L = D^+(\mathcal{G}) - A(\mathcal{G}),
% \end{align}
where $D^+(\mathcal{G})$ is a diagonal matrix with the $i$-th diagonal entry being the sum of all the weights of the out-going arcs incident to node $i$, and $A(\mathcal{G})$ is the {adjacency} matrix of $\mathcal{G}$. 
%
%where $D^+(\mathcal{G})$ is a diagonal matrix with the $i$-th diagonal entry being the sum of all the weights of the out-going arcs incident to node $i$, i.e., $D^{+}_{ii}(\mathcal{G})=\sum_{j=1, j\ne i}^n w_{ij}$, and $A(\mathcal{G})$ is the \textcolor{blue}{ adjacency} matrix of $\mathcal{G}$. 
%i.e. $A_{ij}$ is the weight of the arc $(i,j)$ and $0$ if there is no arc between $i$ and $j$.

In general, the arc weights $w_{ij}$ can be positive or negative, and matrix $L$ is not necessarily symmetric. In practice, however, $w_{ij}$ varies in a small positive range. Fig. \ref{fig: histogram of angles} illustrates the histogram of the angle $\varphi_{ij}$ for all $(i,j)$ in different reduced IEEE standard test cases, where the load flow solution is provided by \textsc{Matpower} \cite{matpower}. Accordingly, $\varphi_{ij} \in (0,\pi)$ in all of these cases. 
%------
%----- trimmed
%This observation stems from the fact that the entries of the admittance matrix, i.e., $Y_{ij} \measuredangle \theta_{ij}$ satisfy the following inequalities: $\theta_{ii} \in [ -\frac{\pi}{2} , 0 ), \: \forall i\in\mathcal{N}$ and $ \theta_{ij} \in [\frac{\pi}{2} , \pi), \: \forall (i,j)\in\mathcal{A}$ where the lower bounds are realized in lossless networks.
%---------------------------------------
% trimmed: 
%Note that there are a number of cases (e.g., IEEE NESTA $9241$-bus system) where $\exists (i,j)\in\mathcal{A}$ such that $\varphi_{ij} \in [-\pi,0]$. This mainly pertains to the transmission lines which have negative impedance and therefore, their $\theta _{ij}$ is close to $-\frac{\pi}{2}$. Note that negative impedance in transmission line data is associated with the equivalent circuit of three-winding transformers. For the sake of simplicity, we focus on the case where the resistance and reactance of all transmission lines are nonnegative.
%-------------------------------------------
%In this case, the EPs are of the form $[\delta^{*\top};\omega^{*\top}]\in \Omega $, where $\Omega$ is the polyhedral set
We make the following reasonable assumption that the EPs of swing equations \eqref{eq: swing equations} are located in the set $\Omega$ defined as
%are of the form $[\delta^{*\top},\omega^{*\top}]\in \Omega $, where $\Omega$ is the polyhedral set
\begin{align*}
\Omega = \left\{ [\delta,\omega]\in\mathbb{R}^{2n} :   0 < \varphi_{ij} < \pi , \forall (i,j) \in \mathcal{A},\omega = 0  \right\}.
\end{align*}
% and 
% \begin{align}
% \Omega= \left\{ [\delta;\omega]\in\mathbb{R}^{2n} :   0 \le \varphi_{ij} \le \pi , \forall (i,j) \in \mathcal{A},\omega = 0  \right\}.
% \end{align}
% \begin{assumption} \label{as: sin positivity}
% The state variable $\delta(t)$ approaches the polyhedral set $\Delta$ as $t\to+\infty$ where 

% \end{assumption}
%------------------------------------
\begin{figure}[t]
 \includegraphics*[width=3.5in, keepaspectratio=true]{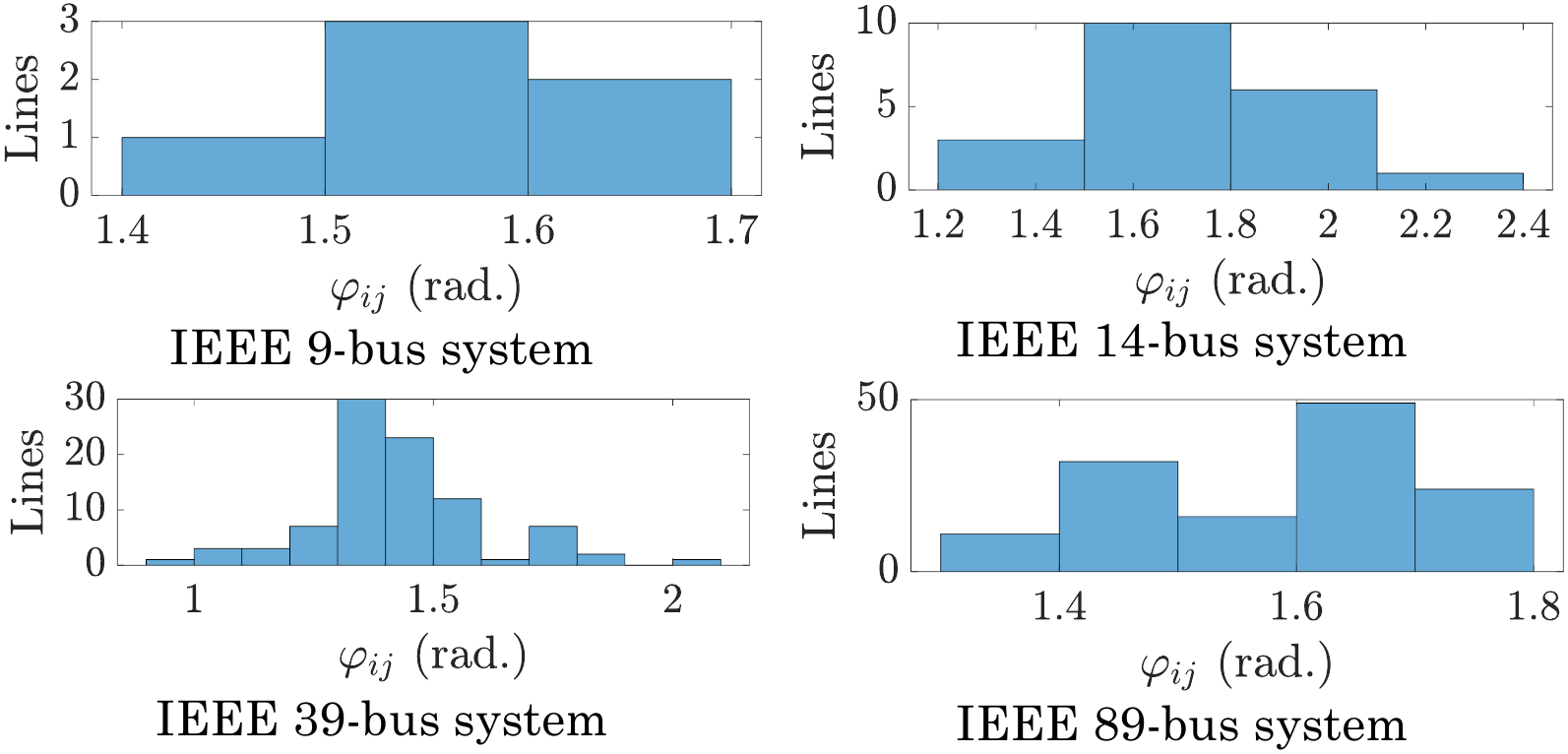}
 \centering
  \caption{Histogram of the distribution of $\varphi_{ij}$ for all $(i,j)$ in different reduced IEEE standard test cases.}
  \label{fig: histogram of angles}
\end{figure}
%-----------------------------------
% \begin{lemma}
% Under Assumption \ref{as: sin positivity}, the Jacobian matrix $L$ is an $L$-matrix.
% \end{lemma}
% \begin{proof}
% Under Assumption \ref{as: sin positivity}, we have 	$\frac {\partial P_i} {\partial \delta_i} > 0 , \: \forall i$ and $\frac {\partial P_i} {\partial \delta_j} \le 0, \: \forall i, j, i \ne j$
% \end{proof}
%Under this assumption, matrix $L$ turns out to be a quasi-$M$ matrix:
\begin{proposition}\label{prop: quasi-M}
Let $[\delta^{*},\omega^{*}]\in \Omega$ be an EP of swing equations \eqref{eq: swing equations}. The Jacobian matrix $L$ at this point is a singular M-matrix. In particular, $L$ has the following properties:
\begin{enumerate}[(i)]
	\item  $L =  \lambda_0 I - B$ for some nonnegative matrix $B$ (i.e., $b_{ij} \ge 0, \forall i, j$) and some $\lambda_0 \ge \rho$, where $\rho$ is a maximal (non-negative) eigenvalue of $B$.
	
	\item All principal minors of $L$ are non-negative.
	
	\item $L$ has at least one zero eigenvalue, $\mathbf{1}$ is an eigenvector, and the real part of each non-zero eigenvalue of $L$ is positive.
\end{enumerate}
%\textbf{I have to prove that an $L$ matrix which is weakly diagonally dominant is a quasi-M matrix }
% I need to prove that J11 is quasi-m matrix. Then I can use Theorem 3.1 of this paper.
%A Survey on M-Matrices
%Author(s): George Poole and Thomas Boullion
%Source: SIAM Review, Vol. 16, No. 4 (Oct., 1974), pp. 419-427
%Published by: Society for Industrial and Applied Mathematics
%Stable URL: https://www.jstor.org/stable/2028688
%Accessed: 29-10-2018 23:07 UTC
\end{proposition}

\begin{proof}
When $[\delta^{*},\omega^{*}]\in \Omega$, we have 	$\frac {\partial P_{e_i}} {\partial \delta_i} \ge 0 , \: \forall i\in\mathcal{N}$ and $\frac {\partial P_{e_i}} {\partial \delta_j} \le 0, \: \forall (i,j)\in\mathcal{A}, i \ne j$. Since $L$ has zero row sum, we have $L\mathbf{1}=0\implies0\in\sigma(L)$. Furthermore, the sum of the absolute values of the nondiagonal entries in the $i$-th row of $L$ is equal to $[L]_{ii}$, that is
\begin{align}
    [L]_{ii} = \sum_{j\ne i} |[L]_{ij}| &&\forall i \in \mathcal{N}.
\end{align}
Let $\mathbb{D}([L]_{ii})$ be a closed disc centered at $[L]_{ii}$ with radius $[L]_{ii}$. According to the Gershgorin circle theorem, every eigenvalue of $L$ lies within at least one of the Gershgorin discs $\mathbb{D}([L]_{ii})$, which are located on the right half plane. This shows (iii). The equivalence of (iii) with (i) and (ii) is a fundamental property of M-matrices \cite{1974-M-matrices}.
\end{proof}

We will use property (iii) of matrix $L$ shown in the above proposition later to prove our main result in the next section. 

\section{Sufficient Condition for the Stability of Swing Equations: A Fast Certificate} \label{Sec: Stability and Hyperbolicity of the Equilibrium Points}
In this section, we present our main result on the stability of the swing equation EPs. 
\begin{theorem} \label{thm: stability properties}
    Let $[\delta^{*},\omega^*]\in \Omega$ be an EP of swing equations \eqref{eq: swing equations}. Suppose all generators have positive damping coefficient and inertia, and the underlying {undirected graph of the power grid is connected}. If condition
        \begin{align} \label{eq: condition for stability}
            \sum \limits_{j \ne i} { V_i V_j Y_{ij} \sin \left( {\theta _{ij} - {\delta _i^*} + {\delta _j^*}} \right) } \le \frac{D_i^2}{2M_i},     \forall i \in \mathcal{N} \tag{\textbf{C}}
        \end{align}
        holds, then the EP is asymptotically stable.
         %the nonzero eigenvalues of $J$ are located in the left half plane, i.e., $\sigma(J)\subset\mathbb{C}_-\cup\{0\}$. In  particular, condition (\textbf{C}) guarantees that the EP is hyperbolic and stable.
\end{theorem}
The proof of Theorem \ref{thm: stability properties} is given in Appendix \ref{proof of thm: stability properties}.
\begin{remark} \label{remark: interpretation of condition C}
    Condition \eqref{eq: condition for stability} provides a practical and efficient way to certify the small-signal stability of the EPs. 
    %
    %Moreover, this condition can be incorporated into the optimal power flow problem as an additional constraint, thereby guaranteeing the small-signal stability of the solution. 
    %This condition enforces an upper bound on the diagonal entries of the Jacobian matrix $L$ (see the proof of Theorem \ref{thm: stability properties}), which is proportional to the square of damping and inverse of inertia. 
    %This is consistent with the intuition that if we increase the damping or decrease the inertia, the stability margin of the system will increase.
    %% 
    {
    The left-hand side of condition \eqref{eq: condition for stability} is closely related to the reactive power output of a generator. Note that at an EP the reactive power injected from bus $i$ into the network is $Q_i = - \sum_{j=1}^n { V_i V_j Y_{ij} \sin ( {\theta _{ij} - {\delta_i^*} + {\delta_j^*}} ) }$. Intuitively, when a generator is supplying more reactive power, the left-hand side of condition \eqref{eq: condition for stability} decreases, and this helps make condition \eqref{eq: condition for stability} satisfied.
    }
\end{remark}

{It is worth mentioning that in \cite{2008-Ishizaki-graph-theory-power-system}, small-signal stability of lossless swing equations is studied. It is shown that if $[\delta^{*},\omega^*]\in \Omega$ is an EP, then the EP is locally asymptotically stable. 
%Moreover, the impact of network topology is examined by analyzing and comparing the Kron-reduced model with the structure-preserving one.
Theorem \ref{thm: stability properties} is a generalization of such results to lossy swing equations. Contrary to the lossless case, we will show in the next section that an EP in lossy networks could be unstable even if it belongs to the set $\Omega$.
}

%---------------
\section{Simulation Results} \label{Sec: Computational Experiments}
%\subsection{Analysis of Sufficient Condition \eqref{eq: condition for stability} for Stability}
In this section, we test the practicality of the assumptions on which Theorem \ref{thm: stability properties} is based. We also show how conservative condition \eqref{eq: condition for stability} is, and how it can be used not only as a fast stability certificate, but also as a quantitative measure of the degree of stability.

Table \ref{tab: diffeent ieee cases} provides the details of testing Theorem \ref{thm: stability properties} and condition \eqref{eq: condition for stability} on different IEEE standard test systems \cite{matpower}. All these systems have a connected underlying graph and nonzero transfer conductances. The second column of Table \ref{tab: diffeent ieee cases} shows the domain of $\varphi_{ij}$ in these test cases. Recall that $\varphi_{ij} = {{\theta _{ij}} - {\delta _i} + {\delta _j}}$ is the argument of the $\sin$ function, and having $\varphi_{ij}\in(0,\pi)$ ensures that an EP $[\delta^{*},\omega^*]$ belongs to the set $\Omega$. As can be seen, this property holds in all test cases of Table \ref{tab: diffeent ieee cases}, and therefore, the assumptions of Theorem \ref{thm: stability properties} hold in a wide variety of practical power systems.

%is a practical and efficient way to certify that an EP is stable and hyperbolic. 
%In this section, we further illustrate this issue
%the key aspects of swing equation EPs 
%using several numerical examples. 
\begin{table}[]
\caption{Illustration of the proposed stability certificate in Theorem \ref{thm: stability properties}.}
\label{tab: diffeent ieee cases}
\begin{tabular}{|c|c|c|c|}
\hline
\textbf{Test  case}   & \textbf{Dom($\varphi_{ij}/\pi$)} & \textbf{Dom($\mathcal{S}_i$)} & \textbf{$|\Re{(\lambda_2)}|$} \\ \hline
IEEE $9$-bus            & $[0.48 , 0.52 ] $          & ${[}-0.79,-0.22{]}    $       & $3.18 $                         \\ \hline
IEEE $14$-bus           & $[0.43 , 0.66 ] $          & ${[}-5.08,-0.03{]}    $       & $2.17 $                         \\ \hline
IEEE $30$-bus           & $[0.36 , 0.66 ] $          & ${[}-12.26,-0.51{]}   $       & $0.75 $                         \\ \hline
IEEE $39$-bus           & $[0.37 , 0.62 ] $          & ${[}-7.73,-0.12{]}    $       & $4.95 $                         \\ \hline
IEEE $89$-bus           & $[0.45 , 0.59 ] $          & ${[}-143.75,1166.9{]} $       & $4.15 $                         \\ \hline
IEEE $89$-bus mod.      & $[0.25 , 0.97 ] $          & ${[}-280.19,-0.49{]}  $       & $4.14 $                         \\ \hline
IEEE $118$-bus          & $[0.42 , 0.63 ] $          & ${[}-241.73,-0.21{]}  $       & $0.11 $                         \\ \hline
IEEE $300$-bus          & $[0.30 , 0.72 ] $          & ${[}-266.99,-3.04{]}  $       & $0.15 $                         \\ \hline
\end{tabular}
\end{table}
%----------------------------------
Next, let us define 
\begin{align*}
    \mathcal{S}_i := \sum \limits_{j \ne i} { V_i V_j Y_{ij} \sin \left( {\theta _{ij} - {\delta _i^*} + {\delta _j^*}} \right) } - \frac{D_i^2}{2M_i},
\end{align*}
and recall that according to condition \eqref{eq: condition for stability} in Theorem \ref{thm: stability properties}, if $\mathcal{S}_i\le0, \forall i\in\mathcal{N}$, then the EP of swing equations is asymptotically stable. The third column of Table \ref{tab: diffeent ieee cases} provides the domain of $\mathcal{S}_i$, i.e., $[\min_i \mathcal{S}_i, \max_i \mathcal{S}_i]$. Accordingly, $\mathcal{S}_i\le0$ holds for all test cases, except the IEEE $89$-bus system. Note that the corresponding EPs in these systems are all stable. While the evaluation of condition \eqref{eq: condition for stability} confirms the stability of EPs in all other cases, it gives an inconclusive answer in the IEEE $89$-bus case. However, here we show how condition \eqref{eq: condition for stability} can be used as a quantitative measure of the degree of stability. The positive values of $\mathcal{S}_i$ in the IEEE $89$-bus system pertain to the bus numbers $6233$, $6798$, $7960$, and $9239$, indicating that the stability of the system can be improved by making $\mathcal{S}_i$ negative in these buses via appropriate corrective actions. Exploring the structure of the system reveals that each of these buses is connected to the rest of the grid through a line with a relatively small resistance. As a corrective action, we change these resistances as follows: $r(659,9239)=6\times 10^{-5} \to 0.5\times 10^{-3}$, $r(659,7960)=6 \times 10^{-5} \to 1\times 10^{-3}$, $r(659,6233)=6\times 10^{-5} \to 2\times 10^{-3}$, and $r(659,6798)=7\times 10^{-5} \to 1.5\times 10^{-3}$, where all the values are in p.u. With this corrective action (which can be implemented through flexible AC transmission system (FACTS)
devices), we will have $\mathcal{S}_i\le0, \forall i\in\mathcal{N}$ and condition \eqref{eq: condition for stability} will hold true, certifying the stability of the system (see the test case IEEE $89$-bus mod. in Table \ref{tab: diffeent ieee cases}). 
Fig. \ref{fig: case1 bifurcation} depicts the spectrum of $J$ in the IEEE $89$-bus system before and after implementing the corrective actions. As can be seen, the magnitude of the imaginary parts of the eigenvalues in $\sigma(J)$ is reduced, and their real parts are mainly moved towards $-\infty$, thereby making the modified system less oscillatory. Evidently, condition \eqref{eq: condition for stability} increased the stability margins of the system.
%Here, condition \eqref{eq: condition for stability} as a measure of stability
%
Finally, $\lambda_2\in\sigma(J)$ denotes the closest nonzero eigenvalue of $J$ to the imaginary axis, and
the fourth column of Table \ref{tab: diffeent ieee cases} depicts this value in different cases.
Note that the proposed stability certificate can be fully parallelized, thereby making it even more reliable and resilient for real-time applications.
%
%
% \begin{figure}
% \begin{subfigure}{0.25\textwidth}
%   \centering
%   \includegraphics[width=\linewidth]{figures/89_eigens.pdf}
%   \caption{base case}
%   \label{fig: sfig1 base case}
% \end{subfigure}%
% \begin{subfigure}{0.25\textwidth}
%   \centering
%   \includegraphics[width=\linewidth]{figures/89_eigen_modified.pdf}
%   \caption{modified}
%   \label{fig:sfig2 modified}
% \end{subfigure}
% \caption{Spectrum of $J$ in the IEEE $89$-bus system.}
% \label{fig: case1 bifurcation}
% \end{figure}
%
%-----------------------------------------------------------
\begin{figure}[t]
\includegraphics[width=3.5in, keepaspectratio=true]{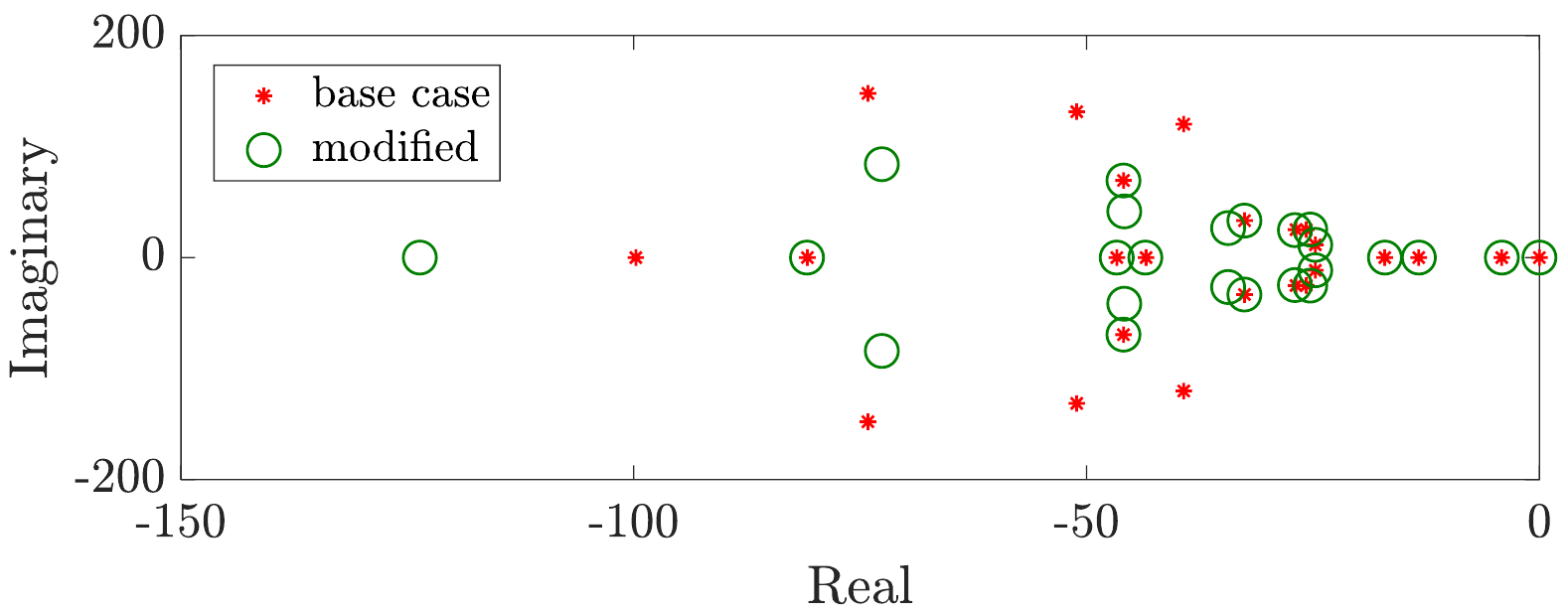}
\centering
\caption{Spectrum of $J$ in the IEEE $89$-bus system.}
\label{fig: case1 bifurcation}
\end{figure}
%---------------------
%---------------------
%-**************************
% Unstable Example
%****************************

Next, we provide an example of an unstable EP and show how enforcing condition \eqref{eq: condition for stability} will make the EP stable. Consider the $3$-bus system in Fig. \ref{fig: schematic diagram of the $3$-bus test system} whose dynamic parameters and converged load flow data are provided in Table \ref{tab: Dynamic Parameters}. 
%Moreover, the schematic diagram of this network is depicted in Fig. \ref{fig: schematic diagram of the $3$-bus test system}. 
%
%
\begin{table} [t]
\centering
\caption{Dynamic parameters and converged load flow data of the $3$-bus test system.}
\label{tab: Dynamic Parameters}
\begin{tabular}{|c|c|c|c|c|c|c|}
\hline
$i$     & $M_i$ [sec.]   & $D_i$  & $P_{m_i}$ [p.u.]    & $V_i$ [p.u.]     & $\delta_i^*$ [rad]  &  $\mathcal{S}_i$   \\ \hline
$1$   & $6.1$       & $1.5$      & $0.89$      & $0.9$                   &  $-0.30$            &   $6.98$           \\ \hline
$2$   & $10$        & $1$        & $15.06$     & $0.9$                   &   $0.36$            &   $12.73$          \\ \hline
$3$   & $4.5$       & $1.8$      & $2.53$      & $0.913$                 &   $-0.12$           &   $8.91$         \\ \hline
\end{tabular}
\end{table}
As can be observed from the last column of Table \ref{tab: Dynamic Parameters}, we have $\mathcal{S}_i\ge0, \forall i\in\mathcal{N}$, i.e., condition \eqref{eq: condition for stability} is violated in all buses of this system, indicating that the system does not have sufficient stability margins. The instability of this EP can be verified through eigenvalue analysis and time domain simulation, as depicted in Fig. \ref{fig: unstable 3 gen}. In order to achieve stability, the power system operator can enforce condition \eqref{eq: condition for stability} either by moving the current EP to a new point (e.g., through adding constraint \eqref{eq: condition for stability} to the optimal power flow problem) or by making the current EP stable through adjusting the right-hand side of condition \eqref{eq: condition for stability}. Particularly, the latter is possible if we have inverter-based resources where the inertia $M_i$ and damping $D_i$ are adjustable parameters of their controllers. In this case, by setting $M=\mathbf{diag}(0.9,0.9,0.9)$ and $D=\mathbf{diag}(4.5,4.9,4.8)$, we would have $\mathcal{S}_1=-4.08$, $\mathcal{S}_2=-0.55$, and $\mathcal{S}_3=-3.52$, thereby certifying the stability of the system.
%------------
\begin{figure}[t]
 \includegraphics*[width=2.6in, keepaspectratio=true]{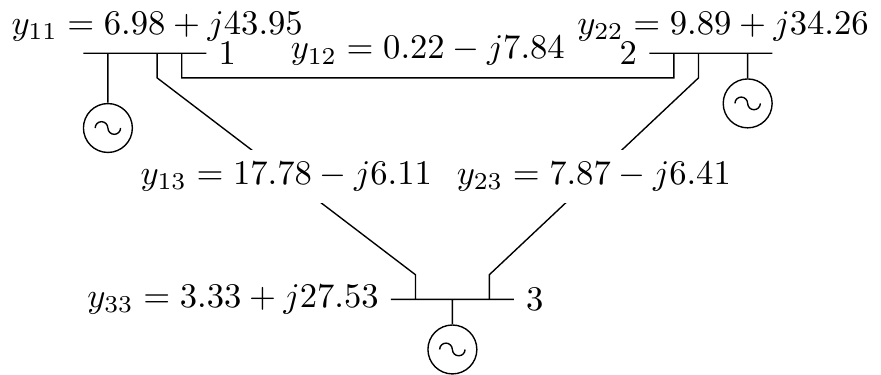}
 \centering
  \caption{Schematic diagram of the $3$-bus test system.}
  \label{fig: schematic diagram of the $3$-bus test system}
\end{figure}
%-----------
\begin{figure}[t]
\begin{subfigure}{0.25\textwidth}
  \centering
  \includegraphics[width=\linewidth]{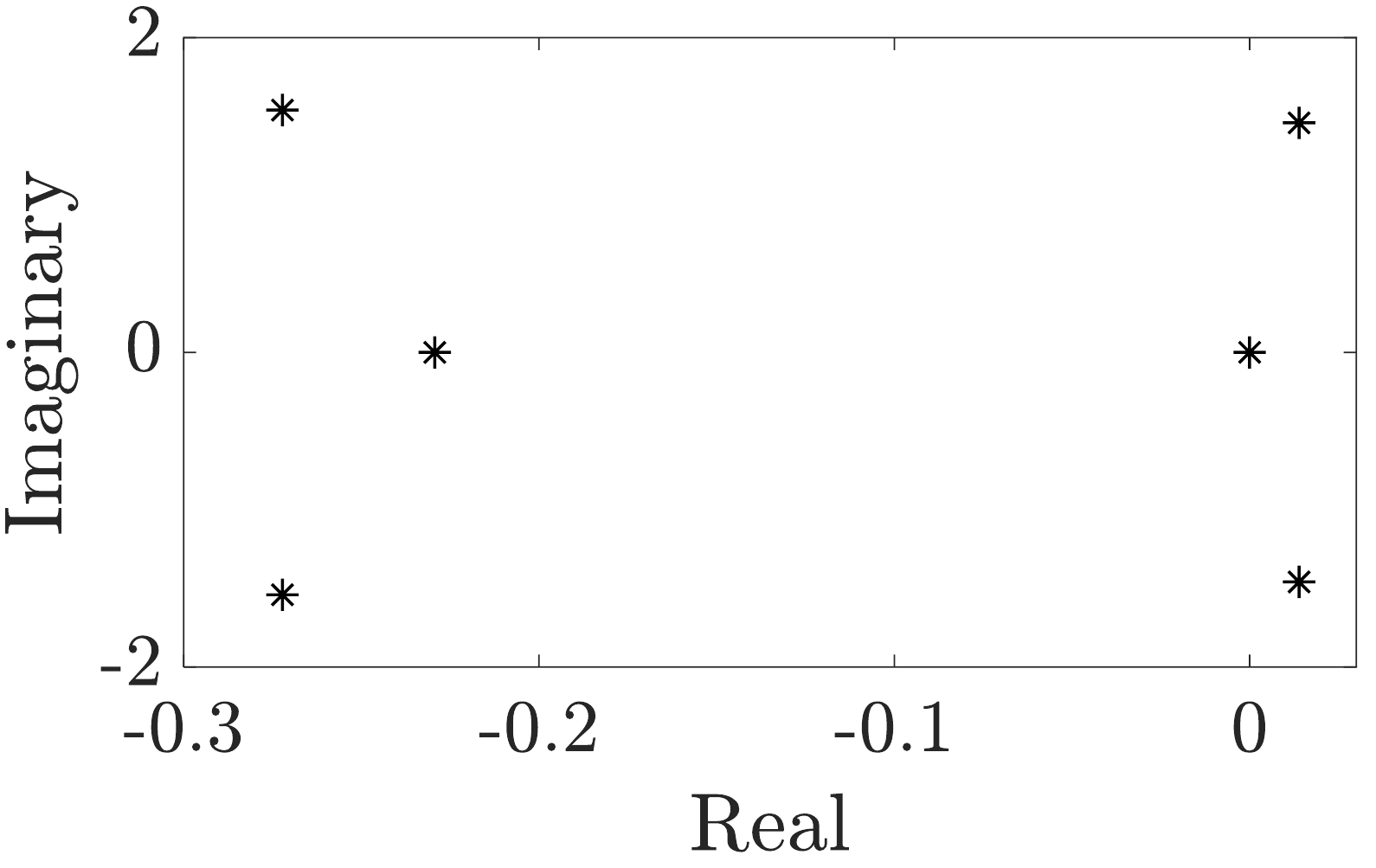}
  \caption{Eigenvalues of matrix $J$.}
  \label{fig: sfig1 unstable eigenvalues}
\end{subfigure}%
\begin{subfigure}{0.22\textwidth}
  \centering
  \includegraphics[width=\linewidth]{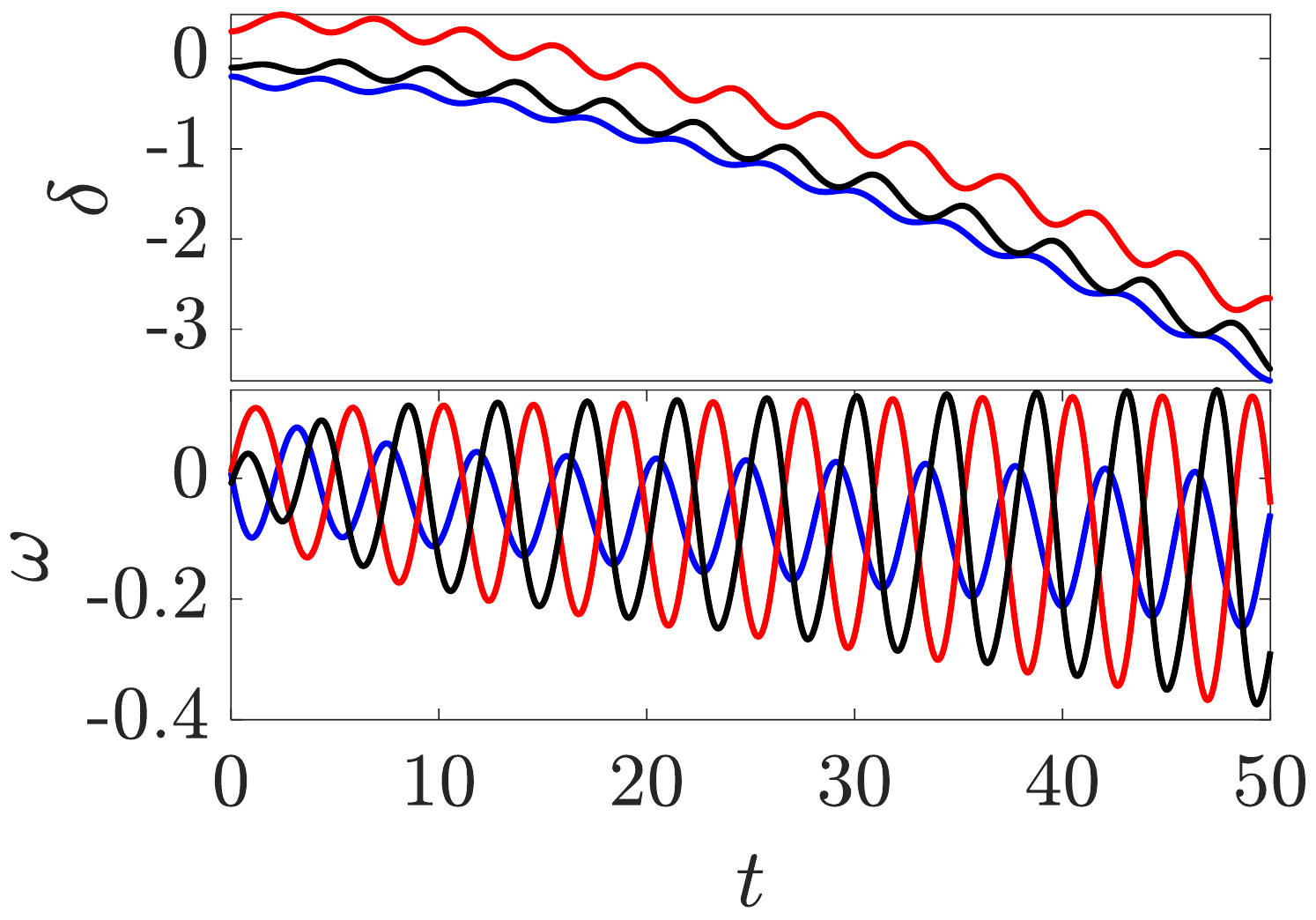}
  \caption{Trajectories of the system.}
  \label{fig: sfig2 unstable time sim}
\end{subfigure}
\caption{Instability of the EP in the $3$-bus test system. (a) There exist two eigenvalues with positive real part. (b) Starting from a neighborhood of the EP, the trajectories become unbounded.}
\label{fig: unstable 3 gen}
\end{figure}

{
We conclude our numerical experiments by further illustrating the effect of condition \eqref{eq: condition for stability} on the spectrum of matrix $J$. We have varied the operating point and parameters (inertia and damping) of the IEEE $9$-bus system, and for each operating point or parameter value we have recorded $\lambda_2$ as well as $\min_i \mathcal{S}_i$. Fig. \ref{fig: condition c qualitative behaviour} shows the relationship between $\lambda_2$ and $\min_i \mathcal{S}_i$ as the system operating point and parameters change. Accordingly, a smaller $\min_i \mathcal{S}_i$ yields a farther $\lambda_2$ from the  imaginary axis.
}
\begin{figure}[t]
 \includegraphics*[width=3.4in, keepaspectratio=true]{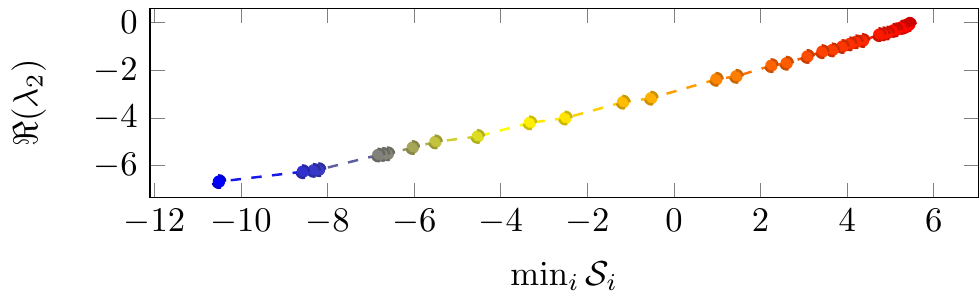}
 \centering
\caption{
{
Real part of the closest nonzero eigenvalue of $J$ to the imaginary axis as a function of $\min_i \mathcal{S}_i$ in the IEEE $9$-bus system.}
  \label{fig: condition c qualitative behaviour}
  }
\end{figure}

%----------------------------------------------------------
\section{Conclusions and Outlook} \label{Sec: Conclusions}
This paper is aimed at finding a computationally efficient way to certify the stability of power system EPs.
%
%the conditions under which the EPs of lossy swing equations are stable.
% We have shown that a quadratic matrix pencil links the Jacobian of swing equations to the Laplacian of the underling graph of the power grid. 
% We have focused on the cases where a swing equation EP is located within a specific set, and the graph of the network is connected (both physically and electrically). Under such circumstances, 
We have shown if the matrix norm of the Laplacian of the underlying graph is upper bounded by a specific value, then the EP is stable. The aforementioned upper bound is proportional to the square of damping and inverse of inertia at each node of the power grid. This fact also sheds light on the interplay of inertia, damping, and graph of the system, and provides profound insights into how power system should be designed and operated to be stable. 
{
A worthwhile direction for future research would be extending condition \eqref{eq: condition for stability} as a function of network connectivity measure.
}
% trimmed: 
%The results developed in this paper can be used to determine the small-signal stability of power system operating points. The developed sufficient conditions can also be incorporated into the optimal power flow problem as an additional constraint in order to guarantee the stability of solutions.  
%----------------------------------
%\vspace{-7pt}
\appendices
\section{Proof of Theorem \ref{thm: stability properties}}
\label{proof of thm: stability properties}
    We complete the proof in three steps:\\
    \textbf{Step 1:} First, we show that the zero eigenvalue of $J$ is simple.
    %the Jacobian $J$ at the EP has a zero eigenvalue of multiplicity one with a one-dimensional invariant subspace $\mathbf{span}(\mathbf{1})$.
    %\begin{proof}[proof (a)]F
     %This is an immediate consequence of Proposition \ref{prop: nullity of J11}.
    According to Proposition \ref{prop: quasi-M}, if $[\delta^{*},\omega^{*}]\in \Omega$, the Jacobian matrix $L$ is a singular M-matrix, and consequently, it has at least one zero eigenvalue.
    Consider the weighted directed graph $\mathcal{G}=(\mathcal{N},\mathcal{A},\mathcal{W})$ constructed in the beginning of Section \ref{subsec: Spectral Properties of the Jacobian Matrix L}. %Let $\mathcal{C}_0$ be a zero cut in this graph and note that $w_{ij}=0, \forall(i,j)\in\mathcal{A}(\mathcal{C}_0)$. Therefore, the existence of a zero cut $\mathcal{C}_0$ would  make the graph electrically disconnected; and the emptiness of the set of zero cuts precludes such an electrical disconnection. 
    {If $[\delta^{*},\omega^{*}]\in \Omega$, the arc weights $w_{ij}$ are positive for all arcs $(i,j)\in\mathcal{A}$. Moreover, there are two arcs $(i,j)$ and $(j,i)$ between nodes $i$ and $j$ if and only if the two nodes are connected in the underlying undirected graph of the power grid. Therefore, if the underlying undirected graph of the power grid is connected, then the directed graph $\mathcal{G}$ is strongly connected. Now, we need the following lemma from graph theory to complete the proof: consider a weighted directed graph $\mathcal{G}$ with positive weights. If $\mathcal{G}$ is strongly connected, then the zero eigenvalue of its Laplacian is simple (see \cite{2018-Dorfler-algebraic-graph-theory} and references therein). Note that the geometric multiplicity of the zero eigenvalue in $\sigma(J)$ and $\sigma(L)$ are equal.
    }
    
    %If the underlying undirected graph of the power grid is connected, then the directed graph $\mathcal{G}$ is weakly connected., then the directed graph $\mathcal{G}$ is weakly connected.
    
    % THIS IS #Removed in R1
    % \textcolor{red}{
    % The connectedness of the graph will be equivalent to the existence of a directed spanning tree in the graph $\mathcal{G}$. Recall that a directed graph has a directed spanning tree if there exists at least one node having a directed path to all other nodes. Now, we need the following lemma from graph theory to complete the proof: consider a weighted directed graph $\mathcal{G}=(\mathcal{N},\mathcal{A},\mathcal{W})$ and let $L$ be its Laplacian matrix where $[L]_{ii}\ge0$, $[L]_{ij}\le0,\forall i\ne j$, and $\sum_{j=1}^n [L]_{ij} = 0$. Then $L$ has exactly one zero eigenvalue if and only if $\mathcal{G}$ has a spanning tree (see Lemma 3.3 in \cite{2005-Ren-Consensus}). This completes the proof.}

    %\end{proof}      
    %--------------
    % \begin{proof}[proof (b)]
    % \textcolor{red}{Modify this definition also}
    % \end{proof}
    %--------------

    \textbf{Step 2:} Next, we show all the nonzero real eigenvalues of $J$ are negative.
    %\begin{proof}[proof (b)]
    Let $\lambda\in\mathbb{R}$ be an eigenvalue of $J$, then according to Lemma \ref{lemma: relation between ev J and ev J11},
            \begin{align} \label{eq: pencil det zero}
                \mathbf{det}\left( L + \lambda D + \lambda^2 M \right)  = 0.
            \end{align}
    Consider the Gershgorin disk $\mathbb{D}_i$ centered at $c_i:=[L]_{ii}+\lambda D_i + \lambda^2 M_i$ with radius $r_i:= [L]_{ii} = \sum_{j\ne i} |[L]_{ij}|$. According to the Gershgorin circle theorem, every eigenvalue of the matrix $L + \lambda D + \lambda^2 M$ lies within at least one of the discs $\mathbb{D}_i, \forall i \in \mathcal{N}$. Now assume for the sake of contradiction that $\lambda>0$, but this implies that $c_i>r_i, \forall i \in \mathcal{N}$, and consequently none of the Gershgorin disks contains the origin (i.e., $0$ cannot be an eigenvalue), contradicting \eqref{eq: pencil det zero}.
    % Secondly, if the eigenvalue $\lambda$ is a real number, then since $L$, $D$, and $M$ are all positive semidefinite, the only way to make $ L  + \lambda D + \lambda^2 M$ singular is to have $\lambda\le 0$. So all the real nonzero eigenvalues of $J$ are negative.
    %\end{proof}
    %--------------

    %\begin{proof}[proof (c)]
    \textbf{Step 3:} Finally, we show if condition \eqref{eq: condition for stability} holds, then the nonzero eigenvalues of $J$ are located in the left half plane. This result holds for real nonzero eigenvalues of $J$, as shown in the previous step.
    Now let $\lambda\in\mathbb{C},  \lambda\in\sigma(J)$, then according to Lemma \ref{lemma: relation between ev J and ev J11}, $\exists v\in\mathbb{C}^n,v\ne0$ such that
            \begin{align} \label{eq: pencil singularity theorem proof}
                \left( L + \lambda D + \lambda^2 M \right) v = 0.
            \end{align}
    It is always possible to normalize $v$ such that $\max_{i\in\mathcal{N}} |v_i| = 1$. {Here and in the rest of this proof, if $x\in\mathbb{C}$, then $|x|$ denotes the modulus of $x$.} Let $k:=  \mathrm{argmax}_{i\in\mathcal{N}} |v_i|$, and spell out the $k$-th row of \eqref{eq: pencil singularity theorem proof}:
            \begin{align} \label{eq: pencil singularity theorem proof k-th row}
                \sum_{i\in\mathcal{N}} [L]_{ki}v_i + \lambda D_k v_k + \lambda^2 M_kv_k   = 0,
            \end{align}
    which can be rewritten as
            \begin{align} \label{eq: pencil singularity theorem proof k-th row rewrite}
                [L]_{kk}v_k + \lambda D_k v_k + \lambda^2 M_kv_k   = -\sum_{i\in\mathcal{N}, i\ne k} [L]_{ki}v_i.
            \end{align}
    Using the triangle inequality, we have
    \begin{align*}
      \bigl\lvert-\sum_{i\in\mathcal{N}, i\ne k} [L]_{ki}v_i \bigr\rvert \le  
      \sum_{i\in\mathcal{N}, i\ne k} \bigl\lvert [L]_{ki}\bigr\rvert \bigl\lvert v_i \bigr\rvert \le 
      \sum_{i\in\mathcal{N}, i\ne k} \bigl\lvert [L]_{ki}\bigr\rvert.
    \end{align*}
    Let us also define $\mathcal{R}:=\sum_{i\in\mathcal{N}, i\ne k} \bigl\lvert [L]_{ki}\bigr\rvert$.
    % \begin{align*}
    %     \mathcal{R}:=\sum_{i\in\mathcal{N}, i\ne k} \bigl\lvert [L]_{ki}\bigr\rvert.
    % \end{align*}
    Now assume that $\lambda=\alpha+j\beta$ with $\alpha\ge0, \beta \ne 0$ is a nonzero eigenvalue of $J$, and let us lead this assumption to a contradiction. Equation \eqref{eq: pencil singularity theorem proof k-th row rewrite} implies that
    %\label{eq: pencil singularity theorem proof k-th row rewrite abs}
    \begin{align*} 
     \mathcal{R}^2 \ge & \bigl\lvert  [L]_{kk}v_k + \lambda D_k v_k + \lambda^2 M_kv_k \bigr\rvert^2 \\
     =&  \bigl\lvert  [L]_{kk} + \lambda D_k  + \lambda^2 M_k \bigr\rvert^2 \bigl\lvert v_k\bigr\rvert^2 \\
     =&   \bigl\lvert  [L]_{kk} + \alpha D_k + (\alpha^2 -\beta^2) M_k +j(2\alpha\beta M_k +\beta D_k)      \bigr\rvert^2 \\
     %
      %=& \textcolor{blue}{ \left( [L]_{kk} + \alpha D_k + (\alpha^2 -\beta^2) M_k \right)^2 + \left( 2\alpha\beta M_k +\beta D_k \right)^2} \\
     % 
     =& [L]_{kk}^2 + (\alpha D_k + (\alpha^2 -\beta^2) M_k)^2 + 2 [L]_{kk} (\alpha D_k +\alpha^2 M_k) \\
      & - 2 [L]_{kk}\beta^2 M_k  + 4\alpha^2\beta^2 M_k^2 +\beta^2 D_k^2 + 4\alpha\beta^2 M_k D_k.
    \end{align*}
    Recall that if $[\delta^{*},\omega^{*}]\in \Omega$, matrix $L$ has zero row sum, i.e., $\mathcal{R} = [L]_{kk}$. By cancelling $\mathcal{R}^2$ and $[L]_{kk}^2$ terms and moving $2[L]_{kk}\beta^2 M_k$ and $\beta^2 D_k^2$ to the left-hand side, we arrive at
    \begin{align} \label{eq: ineq in proof of lossy}
    \notag  \beta^2 (2[L]_{kk} M_k - D_k^2) \ge &  (\alpha D_k + (\alpha^2 -\beta^2) M_k)^2 \\
    \notag  &   + 2 [L]_{kk} (\alpha D_k +\alpha^2 M_k) \\
      &   + 4\alpha^2\beta^2 M_k^2 + 4\alpha\beta^2 M_k D_k.
    \end{align}
    According to our assumption in condition \eqref{eq: condition for stability}, we have $(2[L]_{kk} M_k - D_k^2)\le0$, thus the left-hand side of the inequality \eqref{eq: ineq in proof of lossy} is nonpositive. If $\alpha\ge0$ and $\beta\ne0$, the right-hand side of \eqref{eq: ineq in proof of lossy} would be positive, which is the desired contradiction.
    %Note that as the power network is connected and the set of zero cuts is empty, we have $[L]_{kk}>0$.
    The idea used in this part of the proof was inspired by Skar \cite{1980-Skar-stability-thesis}.
    %------------ Extra explanation
    Note that the simple zero eigenvalue of the Jacobian matrix $J$ stems from the translational invariance of the flow function \eqref{eq: flow function}. As mentioned earlier, we can eliminate this eigenvalue by choosing a reference bus and refer all other bus angles to it. Therefore, the set of EPs $\{ \delta^* + \alpha \mathbf{1}: \alpha \in \mathbb{R} \}$ will collapse into one EP. Such an EP will be asymptotically stable. 


\begin{thebibliography}{1}
\bibitem{2010-Conejo-decision-making}
A. J. Conejo,  M. Carrion, and J. M. Morales, \emph{Decision Making under Uncertainty in Electricity Markets}. Springer, 2010.
%
\bibitem{2004-stability-classification}
P. Kundur \emph{et al.}, ``Definition and classification of power system stability,'' \emph{IEEE Trans. Power Syst.}, vol. 19, no. 2, pp. 1387-1401, May 2004.
%
\bibitem{2016-Xie-multi-microgrid}
Y.~Zhang and L.~Xie, ``A transient stability assessment framework in
power electronic-interfaced distribution systems," \emph{IEEE Trans. Power
Syst.}, vol. 31, no. 6, pp. 5106–5114, Feb. 2016.

% \bibitem{2019-gholami-sun-multi-microgrid}
% A.~Gholami and X.~A. Sun, ``Towards resilient operation of multimicrogrids: an MISOCP-based frequency-constrained approach," \emph{IEEE Trans. Control Netw. Syst.}, vol. 6, no. 3, pp. 925-936, Sep. 2019.

\bibitem{1988-Chiang-stability-regions}
H.~D.~Chiang, M.~W.~Hirsch, and F. F. Wu, ``Stability regions of nonlinear autonomous dynamical systems," \emph{IEEE Trans. Autom. Control}, vol. 33, no. 1, pp. 16-27, Jan. 1988.

\bibitem{1988-Zaborszky-phase-portrait}
J. Zaborszky, G. Huang, B. Zheng, and T. Leung, ``On the phase portrait of a class of large nonlinear dynamic systems such as the power system," \emph{IEEE Trans. Autom. Control}, vol. 33, no. 1, pp. 4-15, Jan. 1988.

\bibitem{2011-Chiang-book-direct-methods}
H. D. Chiang, \emph{Direct Methods for Stability Analysis of Electric Power Systems: Theoretical Foundation, BCU Methodologies, and Applications}. John Wiley \& Sons, 2011.

\bibitem{1985-Varaiya-direct-methods}
P.~Varaiya, F. F. Wu, and R. L. Chen, ``Direct methods for transient stability analysis of power
systems: Recent results," \emph{Proc. IEEE}, vol. 73, no. 12, pp. 1703-1715, Dec. 1985.

%NNNNN
\bibitem{2016-Turitsyn-Lyapunov-Functions}
T. L. Vu and K. Turitsyn, ``Lyapunov functions family approach to transient stability assessment,'' \emph{IEEE Trans. Power Syst.}, vol. 31, no. 2, pp. 1269-1277, Mar. 2016.


\bibitem{1989-Chiang-energy-functions-lossy}
H. D. Chiang, ``Study of the existence of energy functions for power systems with losses," \emph{IEEE
Trans. Circuits Syst.}, vol. 36, no. 11, pp. 1423-1429, Nov. 1989.

\bibitem{1979-Athay-practical-method}
T. Athay, R. Podmore, and S. Virmani, ``A practical method for the direct analysis of transient
stability," \emph{IEEE Trans. Power App. Syst.}, vol. PAS-98, no. 2, pp. 573-584, Mar. 1979.

\bibitem{2005-ortega-nontrivial-transfer-conductances2}
R. Ortega \emph{et al.}, ``Transient stabilization of multimachine power systems with nontrivial transfer conductances," \emph{IEEE Trans. Autom.
Control}, vol. 50, no. 1, pp. 60-75, Jan. 2005.

\bibitem{1980-Skar-stability-thesis}
% S. J. Skar, ``Stability of multi-machine power systems with nontrivial transfer conductances," \emph{SIAM J. Applied Mathematics}, vol. 39, no. 3, pp. 475-491, Dec. 1980.
S.~J.~Skar, ``Stability of power systems and other systems of second order differential equations," Ph.D. dissertation, Dept. Math., Iowa State Univ., Iowa, USA, 1980.

%NNNNN
\bibitem{2012-Dorfler-Kuramoto}
F. Dorfler and F. Bullo, ``Synchronization and transient stability in power networks and nonuniform Kuramoto oscillators," \emph{SIAM J. Control Optimiz.}, vol. 50, no. 3, pp. 1616-1642, 2012.

% \bibitem{1999-Alberto-role-of-reference}
% L. F. Alberto and N. G. Bretas, ``Synchronism versus stability in power systems," Int.
% J. Elec. Power, vol. 21, no. 4, pp. 261-267, 1999.

\bibitem{2019-Milano-Topology-Impact}
F. Ebrahimzadeh, M. Adeen, and F. Milano, ``On the impact of topology on power system transient and frequency stability," in \emph{EEEIC-ICPS Europe}, 2019.




% \bibitem{2006-Chicone-ODE}
% C. Chicone, \emph{Ordinary Differential Equations with Applications}. Springer Science \& Business Media, 2006.

% \bibitem{2008-anderson-stability}
% P.~M.~Anderson and A.~A.~Fouad, \emph{Power System Control and Stability}. John Wiley \& Sons, 2008.

\bibitem{2018-Sauer-Dynamics}
P. W. Sauer, M. A. Pai, and J. H. Chow, \emph{Power System Dynamics and Stability}. John Wiley \& Sons, 2018.

\bibitem{2001-Tisseur-Pencil}
F.~Tisseur and K.~Meerbergen, ``The quadratic eigenvalue problem," \emph{SIAM Review}, vol. 43, no. 2, pp. 235-286, 2001.

\bibitem{matpower}
R.~D.~Zimmerman and C.~E.~Murillo-Sanchez. \emph{MATPOWER}, Version 7.0 (2019). [Online]. Available: \verb+https://matpower.org+

\bibitem{1974-M-matrices}
G. Poole and T. Boullion. ``A survey on M-matrices," \emph{SIAM Review}, vol. 16, no. 4, pp. 419-427, Oct. 1974.

\bibitem{2018-Dorfler-algebraic-graph-theory}
F. Dorfler, J. W. Simpson-Porco, and F. Bullo, ``Electrical networks and algebraic graph theory: models, properties, and applications," \emph{Proc. IEEE}, vol. 106, no. 5, pp. 977-1005, May 2018.

\bibitem{2008-Ishizaki-graph-theory-power-system}
{
T. Ishizaki, A. Chakrabortty, and J. I. Imura, ``Graph-theoretic analysis of power systems," \emph{Proc. IEEE}, vol. 106, no. 5, pp. 931-952, May 2018.
}

% \bibitem{2005-Ren-Consensus}
% W. Ren and R. Beard, ``Consensus seeking in multiagent systems under dynamically changing interaction topologies," \emph{IEEE Trans. Autom. Control}, vol. 50, no. 5, pp. 655–661, 2005.


\end{thebibliography}
\end{document}